\documentclass[12pt]{amsart}

\usepackage{amsfonts}
\usepackage{amssymb}
\usepackage{amsmath}
\usepackage{amsthm}
\usepackage{enumerate}
\usepackage{mathrsfs}

\newcommand{\mathscripty}{\mathscr}

\newcommand{\meet}{\land}
\newcommand{\join}{\lor}
\newcommand{\rs}{\mathord{\upharpoonright}}
\newcommand{\sm}{\setminus}
\newcommand{\sd}{\bigtriangleup}

\newcommand{\cat}{\smallfrown}
\newcommand{\restrict}{\restriction}

\newcommand{\QQ}{\mathbb{Q}}
\newcommand{\RR}{\mathbb{R}}
\newcommand{\CC}{\mathbb{C}}

\newcommand{\SA}{\mathscripty{A}}
\newcommand{\SB}{\mathscripty{B}}

\newcommand{\SF}{\mathscripty{F}}

\newcommand{\SI}{\mathscripty{I}}
\newcommand{\SJ}{\mathscripty{J}}
\newcommand{\SK}{\mathscripty{K}}
\newcommand{\SM}{\mathscripty{M}}

\newcommand{\SP}{\mathscripty{P}}

\newcommand{\ST}{\mathscripty{T}}

\newcommand{\dd}{\mathfrak{d}}

\newcommand{\MA}{\mathrm{MA}}
\newcommand{\OCA}{\mathrm{OCA}}

\newtheorem{theorem}{Theorem}[section]
\newtheorem*{theorem*}{Theorem}

\newtheorem*{proposition*}{Proposition}
\newtheorem{lemma}[theorem]{Lemma}
\newtheorem*{lemma*}{Lemma}
\newtheorem{corollary}[theorem]{Corollary}
\newtheorem*{corollary*}{Corollar}

\newtheorem*{fact*}{Fact}
\theoremstyle{definition}
\newtheorem{definition}[theorem]{Definition}
\newtheorem*{definition*}{Definition}
\newtheorem{claim}[theorem]{Claim}
\newtheorem*{claim*}{Claim}

\newtheorem*{conjecture*}{Conjecture}

\theoremstyle{remark}

\newtheorem*{example*}{Example}
\newtheorem{remark}[theorem]{Remark}
\newtheorem*{remark*}{Remark}

\newtheorem*{note*}{Note}
\newtheorem{question}{Question}
\newtheorem*{question*}{Question}

\newcommand{\set}[2]{\left\{#1\mathrel{}\middle|\mathrel{}#2\right\}}
\newcommand{\seq}[2]{\left\langle#1\mathrel{}\middle|\mathrel{}#2\right\rangle}
\newcommand{\card}[1]{\left| #1 \right|}

\DeclareMathOperator{\ran}{ran}
\DeclareMathOperator{\dom}{dom}
\DeclareMathOperator{\id}{id}
\DeclareMathOperator{\cf}{cf}

\DeclareMathOperator{\cov}{cov}

\DeclareMathOperator{\Fin}{\tt{Fin}}
\DeclareMathOperator{\Ctble}{\tt{Ctble}}

\DeclareMathOperator{\CSN}{CSN}

\title{Automorphisms of $\SP(\lambda)/\SI_\kappa$}
\address{Department of Mathematics, Miami University, Oxford, OH, 45056}
\author{Paul Larson}\thanks{The research of the first author is partially supported by NSF Grant DMS-1201494.}
\email{larsonpb@miamioh.edu}
\author{Paul McKenney}
\email{mckennp2@miamioh.edu}


\begin{document}

\begin{abstract}
  We study conditions on automorphisms of Boolean algebras of the form $\SP(\lambda)/\SI_{\kappa}$ (where $\lambda$ is
  an uncountable cardinal and $\SI_{\kappa}$ is the ideal of sets of cardinality less than $\kappa$) which allow one to
  conclude that a given automorphism is trivial. We show (among other things) that every 
  automorphism of $\SP(2^{\kappa})/\SI_{\kappa^{+}}$ which is trivial on all sets of cardinality $\kappa^{+}$ is
  trivial, and that MA$_{\aleph_{1}}$ implies both that every automorphism of $\SP(\RR)/\Fin$ is trivial on a cocountable
  set and that every automorphism of $\SP(\RR)/\Ctble$ is trivial.
\end{abstract}
	
\maketitle
	
\section{Introduction}

Given a set $X$ and an ideal $\SI$ on $X$, an automorphism of $\SP(X)/\SI$ is said to be trivial if it is induced by a
bijection between sets in $\SP(X) \setminus \SI$.\footnote{See Definition \ref{def:triviality} below for a more precise
formulation.} In 1956, Walter Rudin \cite{Rudin} showed that if the Continuum Hypothesis holds, then the set of
nontrivial automorphisms of $\SP(\omega)/\Fin$ has cardinality $2^{\aleph_{1}}$. Around 1980, Saharon Shelah
\cite{Shelah.PF} showed that consistently all automorphisms of $\SP(\omega)/\Fin$ are trivial. Boban Veli\v ckovi\'c
\cite{Velickovic.OCAA} later proved from $\OCA + \MA_{\aleph_{1}}$, a weak fragment of the Proper
Forcing Axiom, that all automorphisms of $\SP(\lambda)/\Fin$ are trivial, for all infinite cardinals $\lambda$. 
In the same paper, Veli\v ckovi\'c showed that the existence of nontrivial automorphisms of
$\SP(\omega)/\Fin$ is consistent with MA$_{\aleph_{1}}$.
	
The possibilities for automorphisms of structures of the form $\SP(\lambda)/\SI$, for $\lambda$ an uncountable cardinal
and $\SI$ an ideal containing $\Fin$, seem to be much less understood than the case $\lambda = \omega$. For instance, it appears to be unknown whether ZFC
proves that every automorphism of $\SP(\lambda)/\Fin$ is trivial off of a countable subset of $\lambda$, or that 
every automorphism of $\SP(\lambda)/\Ctble$ is trivial. Shelah and
Stepr{\=a}ns \cite{Shelah-Steprans.2} have recently shown, however, that for every $\lambda$ below the least strongly
inaccessible cardinal, every automorphism of $\SP(\lambda)/\Fin$ is trivial off of a subset of $\lambda$ of cardinality
$2^{\aleph_{0}}$.

Many questions about automorphisms of $\SP(\omega_{1})/\Fin$ are closely related to the question (due to Marian
Turzanski, and often called the Katowice Problem) of whether the Boolean algebras $\SP(\omega)/\Fin$ and
$\SP(\omega_{1})/\Fin$ can be isomorphic. There exists such an isomorphism if and only if there is an automorphism of
$\SP(\omega_{1})/\Fin$ which maps the equivalence class of some infinite set to the equivalence class of an infinite set
of a different cardinality (analogous possibilities exist at higher cardinals). We call automorphisms where this does
not happen cardinality-preserving.\footnote{See Definition \ref{def:card-pres} for a more precise formulation.}

In this paper we consider the ideals $\SI_{\kappa} = \{ X \mid |X| < \kappa\}$, for infinite cardinals $\kappa$. We
prove (Theorem \ref{thm:trivial-kappa+}) than a cardinality-preserving automorphism of $\SP(2^\kappa)/\SI_{\kappa^{+}}$
(for any infinite cardinal $\kappa$) which is trivial on all sets of cardinality $\kappa^{+}$ is trivial. Assuming a
weak fragment of Martin's Axiom, we prove (Theorem \ref{thm:sep->trivial}) the analogous result for automorphisms of
$\SP(\RR)/\Fin$ which are trivial on all countable sets. Assuming another fragment of Martin's Axiom, we show (see
Corollary~\ref{cor:zcons}) that every automorphism of $\SP(\RR)/\Ctble$ is
trivial, and also that every automorphism of $\SP(\RR)/\Fin$ is trivial off of a countable set.
	
In Section~\ref{sec:prelims} we prove several lemmas in a general setting that will be useful for work in later
sections.  In Section~\ref{sec:almost-trivial} we discuss almost-trivial automorphisms, and prove Theorem
\ref{thm:trivial-kappa+} (mentioned above).  In Section~\ref{sec:reals} we introduce a weak fragment of Martin's Axiom,
and use it to prove Theorem \ref{thm:sep->trivial}.  Section~\ref{sec:qsets} develops a necessary and sufficient
condition for when an isomorphism between two countable, atomless subalgebras of $\SP(\lambda)/\SI_\kappa$ can extend to
a trivial automorphism, when $\kappa$ has uncountable cofinality. We use this to study automorphisms of
$\SP(\omega_1)/\Ctble$ when an uncountable $Q$-set exists. In Section~\ref{sec:fixpoints}, we give some conditions on
automorphisms which imply the existence of fixed points. In Section~\ref{sec:ladders} we develop a connection between
ladder systems and non-fixed points of automorphisms of $\SP(\lambda)/\SI_{\kappa}$, showing in particular that if there
is a cardinality-preserving (see Definition \ref{def:card-pres}) automorphism of $\SP(\omega_1)/\Fin$ whose set of
ordinal fixed points is nonstationary, then $2^{\aleph_{0}} = 2^{\aleph_1}$.  Finally, Section~\ref{sec:questions}
contains a list of open questions.

\subsection{Notation}
\label{subsec:notation}

We write $A\sim_\kappa B$ to indicate $\card{A\sd B} < \kappa$.  Given $\sigma\in 2^{<\omega}$, we write $N_\sigma$ for
the set $\set{x\in 2^\omega}{\sigma \subset x}$. Given a set $X$ and a cardinal $\kappa$, we write $\SI_\kappa^X$ for
the ideal of subsets of $X$ of cardinality less than $\kappa$.  When there is no chance of confusion, we will drop the
$X$ and just write $\SI_\kappa$.  We write $\Fin$ for $\SI_{\aleph_{0}}$ and $\Ctble$ for $\SI_{\aleph_{1}}$.  If $\SI$
is an ideal on a set $X$ and $A\subseteq X$, we write $[A]_\SI$ for the equivalence class of $A$ in $\SP(X)/\SI$.  When
there is no chance of confusion, we will simply write $[X]$ instead.  We write $[X]_{\SI} \leq [Y]_{\SI}$ to mean that
$X \setminus Y \in \SI$ and $[X]_{\SI} < [Y]_{\SI}$ to mean that $[X]_{\SI} \leq [Y]_{\SI}$ and $[X]_{\SI} \neq
[Y]_{\SI}$.
	
\section{Preliminaries}
\label{sec:prelims}

\begin{definition}
  \label{def:triviality}
  Suppose that $\SI$ and $\SJ$ are ideals on sets $X$ and $Y$ respectively.  A homomorphism $\pi : \SP(X)/\SI \to
  \SP(Y)/\SJ$ is \emph{trivial} if there is a function $f : Y\to X$ such that $\pi([A]_\SI) = [f^{-1}(A)]_\SJ$ for all
  $A\subseteq X$.  Similarly, if $Z\subseteq X$ then we say that $\pi$ is \emph{trivial on $Z$} if there is a function
  $f : Y\to Z$ such that $\pi([A]_\SI) = [f^{-1}(A)]_\SJ$ for all $A\subseteq Z$.
\end{definition}

%In Definition \ref{def:triviality},
One gets equivalent definitions by allowing the domain of $f$ to be a subset of $Y$ with complement in $\SJ$. 
We say that such a function \emph{witnesses} the triviality of $\pi$.  
We use inverse images to describe trivial homomorphisms since these are guaranteed to
preserve the Boolean operations.  The following lemma shows that we can often work with forward images instead.

\begin{lemma}
  \label{lemma:trivial->bijection}
  Let $X$ and $Y$ be sets, let $\kappa$ be an infinite cardinal, and suppose that $f\colon Y \to X$ witnesses that $\pi
  : \SP(X)/\SI_\kappa\to\SP(Y)/\SI_\kappa$ is a trivial isomorphism. Then there are sets $E\subseteq X$ and $F\subseteq
  Y$ with $X\sm E\in\SI_\kappa^X$ and $Y\sm F\in\SI_\kappa^Y$, such that $f$ restricts to a bijection from $F$ to $E$.
  Moreover, $f^{-1} : E\to F$ witnesses that $\pi^{-1}$ is trivial.
\end{lemma}

\begin{proof}
  Suppose that $A = X\sm \ran{f}$ has cardinality $\ge \kappa$.  Then $[A]$ is nonzero, but $\pi([A]) = [f^{-1}(A)]$ is
  zero, contradiction.  Now suppose that there is a set $A\subseteq X$ such that $|f^{-1}(a)| \ge 2$ for all $a\in A$,
  and $\card{A}\ge \kappa$.  Then $f^{-1}(A)$ has cardinality $\ge \kappa$.  Let $f^{-1}(A) = B\cup C$ be a partition
  such that $f''(B) = f''(C) = A$ and $\card{B},\card{C} \ge \kappa$. Then there is no $D$ such that $f^{-1}(D)
  \sim_{\kappa} B$, a contradiction of the fact that $[B]$ is in the range of $\pi$.  Let $E = \set{x\in
  X}{\card{f^{-1}(x)} = 1}$ and $F = f^{-1}(E)$.  It follows that $f$ restricts to a bijection from $F$ to $E$, and
  $\card{X\sm E},\card{Y\sm F} < \kappa$.
	
  For the last part of the lemma, we want to see that for each $A \subseteq Y$, $\pi^{-1}([A]) = [(f^{-1})^{-1}(A)]$,
  i.e., that $\pi([(f^{-1})^{-1}(A)]) = [A]$.  Now, $(f^{-1})^{-1}(A) = f''(A \cap F)$, so we want $\pi([f''(A \cap F)])
  = [A]$. We have that $\pi([f''(A \cap F)]) = [f^{-1}(f''(A \cap F))]$.  Since $f^{-1}(f''(A \cap F)) \cap F = A \cap
  F$, $[f^{-1}(f''(A \cap F))] = [A \cap F]$, which is the same as $[A]$.

%  Now suppose that there is some $A\subseteq F$ such that $\pi^{-1}([A]) \neq [f''(A)]$.  Choose $B\subseteq E$ such
%  that $\pi([B]) = [A]$, and let $C = f''(A)$.  Then $\card{B\sd C}\ge \kappa$, and by choosing between $A$ and its
%  complement in $F$ in the beginning, we may assume that $\card{B\sm C} \ge \kappa$.  Let $D = B\sm C$.  Since $f^{-1}$
%  is injective on $E$, we have $f^{-1}(D)\cap f^{-1}(C) = \emptyset$.  This is a contradiction, since $\card{f^{-1}(D)}
%  \ge \kappa$, $[f^{-1}(D)] = \pi([D]) \le \pi([B]) = [A]$, and $f^{-1}(C) = A$.
	\end{proof}

%  \begin{lemma}
%    \label{lemma:trivial->small-difference}
%    Let $X$ and $Y$ be sets and let $\kappa$ be an infinite cardinal. Let $\pi : \SP(X)/\SI_\kappa\to\SP(Y)/\SI_\kappa$ be a trivial isomorphism, and suppose that $f, g : Y\to X$ are functions such
%    that $\pi([A]) = [f^{-1}(A)]$ for all $A\subseteq X$.  Then $\set{y\in Y}{f(y)\neq g(y)}$ has cardinality $< \kappa$.\footnote{We don't appear to use this lemma.}
%  \end{lemma}
%  \begin{proof}
%    By Lemma~\ref{lemma:trivial->bijection}, we may find $F\subseteq Y$ and $E\subseteq X$ such that $Y\sm F$ and $X\sm
%    E$ both have cardinality $< \kappa$, and $f$ and $g$ restrict to bijections from $F$ to $E$.  Suppose the set $D =
%    \set{y\in E}{f(y)\neq g(y)}$ has cardinality $\ge \kappa$.  Then we may find $S\subseteq D$ with cardinality
%    $\kappa$ such that $f''(S)\cap g''(S) = \emptyset$.  This is a contradiction, since $[f''(S)] = \pi^{-1}[S] =
%    [g''(S)]$.
%  \end{proof}

\begin{definition}
  A Boolean algebra $\SB$ is \emph{$<\kappa$-complete} if every subset $\SA$ of $\SB$ with cardinality $< \kappa$ has a
  least upper bound.
\end{definition}

\begin{remark}
  The Boolean algebra $\SP(X)/\SI_\kappa$ is $<\cf{\kappa}$-complete.  In particular, if $\SA$ is a family of subsets of
  $X$, and $\card{\SA} < \cf{\kappa}$, then $[\bigcup\SA]$ is a least upper bound for the set $\set{[A]}{A\in\SA}$.
\end{remark}

It is a well-known open question (asked by Marian Turzanski) whether the Boolean algebras $\SP(\omega_{1})/\Fin$ and
$\SP(\omega)/\Fin$ can consistently be isomorphic. In many of the arguments in this paper, we must allow for the
possibility that such an isomorphism exists (as well as analogous isomorphisms at other cardinals). We record here some
facts that we use to deal with this possibility. The following theorem was proved by Balcar and Frankiewicz
\cite{Balcar-Frankiewicz} in the case $\lambda = \omega$ and $\mu = \omega_{1}$; their proof gives the general version
below.

\begin{theorem}(Balcar, Frankiewicz)
  \label{thm:balcar-frankiewicz}
  Suppose $\kappa \le \lambda < \mu$, and $\kappa$ is regular.  If $\SP(\lambda)/\SI_\kappa$ and $\SP(\mu)/\SI_\kappa$
  are isomorphic, then $\lambda = \kappa$ and $\mu = \kappa^+$.
\end{theorem}

We will also make use of the following fact, where $\dd$ is the minimal cardinality of a set $X$, consisting of
functions from $\omega$ to $\omega$, such that every such function is dominated everywhere by a member of $X$.

\begin{theorem}(Balcar, Frankiewicz)
  \label{thm:balcar-frankiewicz:dd}
  If $\SP(\omega)/\Fin$ and $\SP(\omega_1)/\Fin$ are isomorphic, then $\dd = \omega_1$.
\end{theorem}

Finally, we make the following definition.

\begin{definition}\label{def:card-pres}
  A homomorphism $\pi : \SP(X)/\SI\to \SP(Y)/\SJ$ is \emph{cardinality-preserving} if for every $A\subseteq X$, there is
  some $B\subseteq Y$ such that $\card{A} = \card{B}$ and $\pi([A]) = [B]$. 
\end{definition}

\begin{remark}
  An isomorphism $\pi : \SP(X)/\SI_\kappa\to \SP(Y)/\SI_\kappa$ is cardinality-preserving if and only if for all
  $A\subseteq X$, and $B\subseteq Y$, if $\card{A},\card{B}\ge\kappa$ and $\pi([A]) = [B]$ then $\card{A} = \card{B}$.
  By Theorem \ref{thm:balcar-frankiewicz}, for any pair of infinite cardinals $\kappa < \lambda$, there
  exists an automorphism of $\SP(\lambda)/\SI_{\kappa}$ which is not cardinality-preserving if and only if there is an isomorphism
  between $\SP(\kappa)/\SI_{\kappa}$ and $\SP(\kappa^{+})/\SI_{\kappa}$. 
\end{remark}

In our first application of the notion of cardinality-preservation, we show that it allows one to lift automorphisms on
Boolean algebras of the form $\SP(\lambda)/\SI_\kappa$ to ones of the form $\SP(\lambda)/\SI_{\mu}$, when
$\kappa \le \mu \le \lambda$.

\begin{definition}
  Suppose that $\SI$ and $\SJ$ are ideals on sets $X$ and $Y$ respectively, and that $\pi : \SP(X)/\SI \to \SP(Y)/\SJ$
  is a function. A \emph{selector} for $\pi$ is a map $\pi^{*} \colon \SP(X) \to \SP(Y)$ such that $\pi([A]) =
  [\pi^{*}(A)]$ for all $A \subseteq X$.
\end{definition}

\begin{remark}
  A \emph{selector}, in the literature, often denotes a function which is constant on equivalence classes.  Our
  definition does not make this requirement, and in fact we will often instead take selectors which form bijections
  between equivalence classes.
\end{remark}

\begin{lemma}
  \label{lemma:additivity}
  Let $\kappa \le \mu \le \lambda$ be infinite cardinals, and let $\pi$ be an automorphism of
  $\SP(\lambda)/\SI_{\kappa}$. Suppose that at least one of the following holds:
	\begin{itemize}
	  \item $\pi$ is cardinality-preserving;
	  \item $\mu > \kappa^{+}$ and $\kappa$ is regular.
	\end{itemize}
  Then $\pi$ induces an automorphism $\pi_\mu$ of $\SP(\lambda)/\SI_\mu$.  In particular, if $\SA$ is a family of
  subsets of $\lambda$, $\card{\SA} < \cf{\mu}$, and $\pi^*$ is a selector for $\pi$, then
  \[
    \pi^*\left(\bigcup\SA\right) \sim_\mu \bigcup\set{\pi^*(A)}{A\in\SA}
  \]
\end{lemma}

\begin{proof}
  Our assumptions on $\pi$ (using Theorem \ref{thm:balcar-frankiewicz} in the case where $\mu > \kappa^{+}$ and
  $\kappa$ is regular) imply that $\pi$ takes the subalgebra $\SI_\mu / \SI_\kappa$ into itself.  It follows that if
  $\pi^*$ is a selector for $\pi$, then the map \[\pi_\mu([A]_\mu) = [\pi^*(A)]_\mu\] is well-defined and an automorphism
  of $\SP(\lambda)/\SI_\mu$.  The rest follows from the $<\cf{\mu}$-completeness of the Boolean algebra
  $\SP(\lambda)/\SI_\mu$.
\end{proof}

%  \begin{remark}
%    \label{rmk:additivity}
%    Let $\pi$ be as in Lemma~\ref{lemma:additivity}.  Then, since $\SP(\lambda)/\SI_\mu$ is $<\mu$-complete, if $\pi^*$
%    is a selector for $\pi$, $\nu < \mu$, and $A_\delta$ ($\delta < \nu$) is a sequence of subsets of $\lambda$, we have
%    \[
%      \pi^*(\bigcup\set{A_\delta}{\delta < \nu}) \sim_{\mu} \bigcup\set{\pi^*(A_\delta)}{\delta < \nu}
%    \]
%  \end{remark}

 % \begin{definition}

 % \end{definition}

The function $\pi_{\mu}$ from the proof of Lemma \ref{lemma:additivity} clearly does not depend on the choice of $\pi^{*}$. 
We make the following definition, which will be used in Section \ref{sec:almost-trivial}. 

\begin{definition}
\label{definition:pimu}
  Let $\kappa \le \mu \le \lambda$ be infinite cardinals, and let $\pi$ be an automorphism of
  $\SP(\lambda)/\SI_{\kappa}$. We let $\pi_{\mu}$ be the function on $\SP(\lambda)/\SI_{\mu}$ defined by setting 
  $\pi_{\mu}([A]_{\mu}) = [\pi^{*}(A)]_{\mu}$, for each $A \subseteq \lambda$ and any selector $\pi^{*}$ for $\pi$. 
\end{definition}

%We will continue to use the notation $\pi_{\mu}$ as in Lemma \ref{lemma:additivity} in Section \ref{sec:almost-trivial}.

%Lemmas \ref{lemma:sups} and \ref{lemma:equal-on-kappa} each show that every automorphism of a Boolean algebra of the
Lemma~\ref{lemma:equal-on-kappa} shows that every automorphism of a Boolean algebra of the
form $\SP(\lambda)/\SI_{\kappa}$ is determined by how it acts on sets of cardinality $\kappa$.

%  \begin{lemma}
%    \label{lemma:sups}
%    Let $\kappa \le \lambda$ and let $\pi$ be an automorphism of $\SP(\lambda)/\SI_\kappa$.\footnote{We should remove this lemma, since we don't use it.}  Let $\pi^*$ be a selector for
%    $\pi$.  Suppose $X\subseteq\lambda$, and $\SA$ is a collection of subsets of $X$ which is cofinal in
%    $[X]^{\le\kappa}$; then
%    \[
%      \card{\pi^*(X)\sm \bigcup \set{\pi^*(Y)}{Y\in\SA}} < \kappa
%    \]
%  \end{lemma}
%
%  \begin{proof}
%    Let $Z = \pi^*(X)\sm \bigcup \set{\pi^*(Y)}{Y\in\SA}$, and choose $W\subseteq X$ such that $\pi^*(W)\sim_\kappa Z$.
%    Suppose $\card{W}\ge \kappa$; then we may choose $Y\in\SA$ such that $\card{Y\cap W}\ge \kappa$.  Then $\card{Z\cap
%    \pi^*(Y)} \ge \kappa$, which is a contradiction.  Hence $\card{W} < \kappa$, which implies $\card{Z} < \kappa$.
%  \end{proof}

\begin{lemma}
  \label{lemma:equal-on-kappa}
  Let $\kappa \le \lambda$ and let $\pi$ and $\rho$ be automorphisms of $\SP(\lambda)/\SI_\kappa$.  Then if
  $\pi\neq\rho$, there is some $X\subseteq\lambda$ of cardinality $\kappa$ such that $\pi([X]) \neq \rho([X])$.
  Moreover, for each $X \subseteq \lambda$ such that $\pi([X]) \not\leq \rho([X])$, there exists $Y \in [X]^{\kappa}$
  such that $\pi([Y]) \not \leq \rho([Y])$.
\end{lemma}
	
\begin{proof}
  By composing with $\rho^{-1}$, we may assume that $\rho = \id$.  Fix a bijective selector $\pi^*$ for $\pi$, and
  choose $X\subseteq\lambda$ such that $\pi([X])\neq [X]$.  Without loss of generality, by choosing between $X$ and its
  complement we may assume $\pi^*(X)\sm X$ has cardinality $\ge\kappa$. Let $$W = (\pi^{*})^{-1}(\pi^*(X)\sm X).$$ Then
  $|W| \geq \kappa$ and $|W \setminus X| < \kappa$. Let $Y$ be a subset of $W \cap X$ of cardinality $\kappa$.  Then
  $\pi^{*}(Y)$ has cardinality at least $\kappa$, and \[|\pi^{*}(Y) \setminus (\pi^{*}(X)\setminus X)| < \kappa.\] It
  follows that $|\pi^{*}(Y) \setminus Y| \geq \kappa$, so $\pi([Y]) \not\leq [Y]$.
\end{proof}
		
		%Let $\SA$ be the collection of subsets of $X$
    %with cardinality $\kappa$.  By Lemma~\ref{lemma:sups}, it follows that there is some $Y\in\SA$ such that
    %$\pi^*(Y)\sm X$ has cardinality $\ge\kappa$, and hence $\pi([Y])\neq [Y]$.

\section{Almost-trivial automorphisms of $\SP(\lambda)/\SI_\kappa$}
\label{sec:almost-trivial}

\begin{definition}
  Given an automorphism $\pi$ of $\SP(\lambda)/\SI_\kappa$, we define $\ST(\pi)$ to be the ideal of subsets $A$ of
  $\lambda$ such that $\pi$ is trivial on $A$.   We let $\SA_\mu(\pi)$ be the ideal generated by $\ST(\pi)$ and
  $\SI_\mu$.

  If $A\in \SA_\mu(\pi)$, then we say $\pi$ is \emph{$\mu$-almost trivial} on $A$.  If $\pi$ is $\mu$-almost trivial on
  $\lambda$, we just say that $\pi$ is \emph{$\mu$-almost trivial}.  If $\pi$ is $\kappa^+$-almost trivial on a set $A$,
  then we just say that $\pi$ is almost trivial on $A$.
\end{definition}

In Lemma \ref{lemma:bootstrap} we show that if two automorphisms lift to the same automorphism, then they are the same
off of a small set.

\begin{lemma}
  \label{lemma:bootstrap}
  Suppose that $\kappa < \mu \le \lambda$ are infinite cardinals, with $\kappa$ and $\mu$ regular, and let $\pi$ and
  $\rho$ be automorphisms of $\SP(\lambda)/\SI_\kappa$. Suppose that either $\mu > \kappa^{+}$ or that both $\pi$ and
  $\rho$ are cardinality-preserving. If $\pi_\mu = \rho_\mu$ then there is some $A\in\SI_\mu$ such that for all
  $X\subseteq\lambda\sm A$, $\pi([X]) = \rho([X])$.
\end{lemma}

\begin{proof}
  By composing with $\rho^{-1}$, we may assume that $\rho = \id$.  Let $\pi^*$ be a bijective selector for $\pi$. Our
  assumptions on $\pi$ imply that $\pi$ is a permutation of $\SI_{\mu}$.  Suppose that the conclusion of the lemma
  fails, so that for every $A\in\SI_\mu$ there is some $X\subseteq\lambda$ disjoint from $A$ with $\pi([X])\neq [X]$.
  We will show that $\pi_\mu$ is not the identity.

  Fix some $A\in\SI_\mu$.  By our assumption, there is some $X$ disjoint from $A \cup \pi^{*}(A) \cup (\pi^{*})^{-1}(A)$
  with $\pi^*(X)\not\sim_{\kappa} X$.  By choosing between $X$ and $\lambda\sm (A\cup X)$, we may assume that
  $\card{\pi^*(X)\sm X} \ge \kappa$.  By Lemma~\ref{lemma:equal-on-kappa}, we may also assume that $X$ has cardinality
  $\kappa$.  Applying this observation repeatedly, we may construct sets $A_\alpha$ ($\alpha < \mu$) in $\SI_\mu$, and
  sets $X_\alpha$ ($\alpha < \mu$) of cardinality $\kappa$, such that
  \begin{itemize}
    \item  for all $\alpha < \mu$,
	  \begin{itemize}
			\item $X_\alpha\cap (A_\alpha \cup \pi^{*}(A) \cup (\pi^{*})^{-1}(A)) = \emptyset$,
			\item $\card{\pi^*(X_\alpha)\sm X_\alpha} \ge
      \kappa$,
			\item $A_\alpha\cup X_\alpha\cup \pi^*(A_\alpha)\cup (\pi^*)^{-1}(A_\alpha) \cup \pi^*(X_\alpha)\subseteq
      A_{\alpha+1}$,
		\end{itemize}
    \item  for all limit $\alpha < \mu$, $A_\alpha = \bigcup\set{A_\beta}{\beta < \alpha}$.
  \end{itemize}
  For each $\alpha < \mu$, $\card{\pi^*(X_{\alpha})\cap A_\alpha} < \kappa$, so $\card{\pi^*(X_\alpha)\sm
  (A_{\alpha+1}\sm A_\alpha)} < \kappa$.  Let $X = \bigcup\set{X_\alpha}{\alpha < \mu}$.  Since, for all $\alpha < \mu$,
  we have
  \begin{gather*}
    \card{(\pi^*(X_\alpha)\sm X_\alpha)\cap (A_{\alpha+1}\sm A_\alpha)} \ge \kappa \\
    \card{\pi^*(X_\alpha)\sm \pi^*(X)} < \kappa \\
    X \cap (A_{\alpha+1}\sm A_\alpha) = X_\alpha
  \end{gather*}
  it follows that $\card{(\pi^*(X)\sm X)\cap (A_{\alpha+1}\sm A_\alpha)} \ge \kappa$ for every $\alpha < \mu$.  Then
  $\card{\pi^*(X)\sm X}\ge \mu$.  This completes the proof.
\end{proof}

By applying Lemma~\ref{lemma:bootstrap} in the case where $\pi_\mu$ (recall Definition \ref{definition:pimu}) is trivial, we obtain the following
\begin{theorem}
  \label{thm:lifts}
  Suppose that $\kappa < \mu \le \lambda$ are infinite cardinals, with $\kappa$ and $\mu$ regular, and let $\pi$ be an
  automorphism of $\SP(\lambda)/\SI_\kappa$. Suppose that either $\pi$ is cardinality-preserving or $\mu > \kappa^{+}$.
  If $\pi_\mu$ is trivial, then $\pi$ is $\mu$-almost trivial.
\end{theorem}

Theorem \ref{thm:trivial-kappa+} is one of the main results of the paper.  The strategy used in its proof is reused in
Section \ref{sec:reals}.

\begin{theorem}
  \label{thm:trivial-kappa+}
  Suppose that $\pi$ is an automorphism of $\SP(2^\kappa)/\SI_{\kappa^+}$, and that $\pi$ is trivial on
  every set of cardinality $\kappa^+$.  Then $\pi$ is trivial.
\end{theorem}

\begin{proof}
  Let $\pi^*$ be a bijective selector for $\pi$, and for each $A\subseteq 2^\kappa$ of cardinality $\kappa^+$, choose a
  function $f_A : A\to \pi^*(A)$ such that for all $B\subseteq A$, $\pi^*(B) \sim_{\kappa^+} f_A''(B)$.  By Lemma
  \ref{lemma:trivial->bijection}, each $f_{A}$ restricts to a bijection between subsets of $A$ and $\pi^*(A)$ whose
  complements (in $A$ and $\pi^*(A)$ respectively) have cardinality at most $\kappa$, and moreover, for every
  $B\subseteq\pi^*(A)$, $\pi^{-1}([B]) = [f_A^{-1}(B)]$.

  Let $\seq{x_\alpha}{\alpha < 2^\kappa}$ be an enumeration of $\SP(\kappa)$. For each $\beta < \kappa$, let
  $R_{\beta} = \set{ \alpha < 2^{\kappa} }{\beta \in x_{\alpha}}$ and let $T_{\beta} = (\pi^{*})^{-1}(R_{\beta})$. For
  each $\gamma < 2^{\kappa}$, let $y_{\gamma} = \set{ \beta < \kappa }{ \gamma \in T_{\beta}}$. Let $h \colon
  2^{\kappa} \to 2^{\kappa}$ be such that for all $\gamma, \alpha < 2^{\kappa}$, if $y_{\gamma} = x_{\alpha}$, then
  $h(\gamma) = \alpha$.\footnote{While it is not important for the current proof, we note (without any triviality
  condition on $\pi$) that $h''(A) \sim_{\kappa^{+}} \pi^{*}(A)$ for every $A$ in the smallest $\kappa$-complete
  subalgebra of $\SP(2^{\kappa})$ containing $\SI_{\kappa^{+}}$ and the sets $T_{\beta}$ $(\beta < \kappa)$.}

%  Fix $\beta < \kappa$ and $A\in [2^\kappa]^{\kappa^+}$.  Then,
%  \[
%    f^{-1}_A[R_\beta] \sim_{\kappa^+} (\pi^*)^{-1}(R_\beta\cap \pi^*(A)) \sim_{\kappa^+} T_\beta\cap A
%  \]
%  which implies that the set
%  \[
%    G_{A,\beta} = \set{\gamma\in A}{\gamma\in T_\beta \iff f_A(\gamma)\in R_\beta}
%  \]
%  differs from $A$ by a set of cardinality at most $\kappa$.

  For each $\beta < \kappa$, and $A\in [2^\kappa]^{\kappa^+}$, let $G_{A, \beta}$ be the set of $\gamma \in A$
  for which $\gamma \in T_\beta$ if and only if $f_{A}(\gamma) \in R_{\beta}$.  Since for each such $\beta$ and $A$, we
  have
  \[
    f^{-1}_{A}[R_{\beta}] \sim_{\kappa^{+}} (\pi^{*})^{-1}(R_{\beta} \cap \pi^{*}(A)) \sim_{\kappa^+} (T_{\beta} \cap A),
  \]
  it follows that $G_{A,\beta} \sim_{\kappa^+} A$.  Then, for each $A\in [2^\kappa]^{\kappa^+}$,
  \[
    H_A = \bigcap\set{G_{A,\beta}}{\beta < \kappa} \sim_{\kappa^+} A.
  \]
  For each $\gamma\in H_A$ and $\beta < \kappa$, we have
  \[
    \beta\in y_\gamma \iff \gamma\in T_\beta \iff f_A(\gamma)\in R_\beta \iff \beta\in x_{f_A(\gamma)}.
  \]
  Then, for each $\gamma\in H_A$, $y_\gamma = x_{f_A(\gamma)}$, so $h(\gamma) = f_A(\gamma)$.  Since
  $H_A\sim_{\kappa^+} A$, we then have $h\rs A \sim_{\kappa^+} f_A$.

%  Now note that $\gamma\in T_\beta$ if and only if $\beta\in
%  y_\gamma$, 
%  \[
%    \gamma\in G_{A,\beta} \iff (\beta
%  which means that $\card{A \setminus G_{A, \beta}} \le \kappa$.  Furthermore, for any such $A$, and any $\gamma \in A$,
%	$h(\gamma) = f_{A}(\gamma)$ if $y_{\gamma} = x_{f_{A}(\gamma)}$, i.e., if for all $\beta < \kappa$,
%  $\gamma \in G_{A, \beta}$. It follows that for any such an $A$ of cardinality $\kappa^{+}$, $f_{A} \sim_{\kappa^{+}}
%  (h \restriction A)$.

  It follows from this that \[\card{\set{ \alpha < 2^{\kappa} }{ \card{h^{-1}[\{\alpha\}]} \neq 1}} \le \kappa,\] so for
  some $B,C \in [2^{\kappa}]^{\leq \kappa}$, $h \restriction (2^{\kappa} \setminus B)$ is a bijection between
  $(2^{\kappa} \setminus B)$ and $(2^{\kappa} \setminus C)$.  Thus, the map $\rho([A]) = [h''(A)]$ defines a trivial
  automorphism of $\SP(2^\kappa)/\SI_{\kappa^+}$.  By the above, we have $\pi([A]) = \rho([A])$ for all $A\subseteq
  2^\kappa$ with cardinality $\le \kappa^+$; hence, by Lemma~\ref{lemma:equal-on-kappa}, $\pi = \rho$.

%    Assuming that $\pi$ is not trivial, there exist pairwise disjoint sets $D_{\delta} \subseteq (2^{\kappa} \setminus
%    B)$ of cardinality $\kappa$ such that, for each $\delta$, $h[D_{\delta}] \not\sim_{\kappa} \pi^{*}(D_{\delta})$, so,
%    in particular, $(h \restrict D_{\delta}) \not\sim_{\kappa} f_{D_{\delta}}$. Let $D = \bigcup\{ D_{\delta} : \delta <
%    \kappa^{+}\}$. Then $(h \restrict D) \sim_{\kappa^{+}} f_{D}$, and, for all $\delta < \kappa^{+}$, $(f_{D} \restrict
%    D_{\delta}) \sim_{\kappa} f_{D_{\delta}}$, giving a contradiction.
\end{proof}

\begin{remark}
  Theorem \ref{thm:trivial-kappa+} contradicts the remark at the end of \cite{Velickovic.OCAA} which claims that
  MA$_{\aleph_{1}}$ + OCA (which implies that $2^{\aleph_{0}} \ge \aleph_{2}$, and that all automorphisms of
  $\SP(\omega_{1})/\Fin$ are trivial) does not imply that all automorphisms of $\SP(\omega_{2})/\Fin$ are trivial.
  Combining Theorem \ref{thm:trivial-kappa+} with the main result of \cite{Shelah-Steprans.2}, one gets that if all
  automorphisms of $\SP(\omega_{1})/\Fin$ are trivial, then all automorphisms of $\SP(\lambda)/\Fin$ are trivial, for all
  $\lambda$ below the least strongly inaccessible cardinal.
\end{remark}

\begin{remark}
  Theorems \ref{thm:lifts} and \ref{thm:trivial-kappa+} show that if $\mu < \kappa$ are infinite cardinals and $\pi$ is
  an automorphism of $\SP(2^{\kappa})/\SI_{\mu}$ which is trivial on all sets of cardinality $\kappa^{+}$, then $\pi$ is
  trivial.
\end{remark}

We finish this section with facts about $\ST(\pi)$ which will be used in Section \ref{sec:reals}.

\begin{lemma}
  \label{lemma:closure}
  Suppose that $\kappa \leq \lambda$ are infinite cardinals, and that $\pi$ is an automorphism of
	$\SP(\lambda)/\SI_\kappa$.
  Then $\ST(\pi)$ is closed under unions of cardinality less than $\cf\kappa$.
\end{lemma}

\begin{proof}
  Fix a cardinal $\gamma < \cf\kappa$, and let $A_\delta$ ($\delta < \gamma$) be sets in $\ST(\pi)$.  We may assume that
  each $A_\delta$ has cardinality $\ge \kappa$, and that the $A_\delta$'s are pairwise disjoint.  Applying Lemma
  \ref{lemma:trivial->bijection} (and possibly removing a set of cardinality less than $\kappa$ from each $A_{\delta}$)
  let $f_\delta : A_\delta\to \lambda$ ($\delta < \gamma$) be injections such that for all $\delta < \gamma$ and
  $X\subseteq A_\delta$, $\pi([X]) = [f_\delta''(X)]$.  Put $f = \bigcup\set{f_\delta}{\delta < \gamma}$ and $A =
  \bigcup\set{A_\delta}{\delta < \gamma}$.  Fix a selector $\pi^*$ of $\pi$, and let $X\subseteq A$.  Since
  $\SP(\lambda)/\SI_\kappa$ is $<\cf\kappa$-complete and $\pi$ is an automorphism, it follows that
  \[
    \pi^*(X) \sim_{\kappa} \bigcup\set{\pi^*(X\cap A_\delta)}{\delta < \gamma}.
  \]
  Note also that
  \[
    f''(X) = \bigcup\set{f''(X\cap A_\delta)}{\delta < \gamma}.
  \]
  Since $\pi^*(X\cap A_\delta) \sim_{\kappa} f_\delta''(X\cap A_\delta)$ for every $\delta < \gamma$, $\pi^*(X)
  \sim_{\kappa} f''(X)$.  This shows that $A\in\ST(\pi)$, as required.
\end{proof}

Theorem \ref{thm:lifts}  and Lemma \ref{lemma:closure} give the following.

\begin{lemma}
  \label{lemma:closure2}
  Suppose that $\kappa < \lambda$ are infinite cardinals, and that $\pi$ is a cardinality-preserving automorphism of
  $\SP(\lambda)/\SI_\kappa$.  Then $\SA_{\kappa^{+}}(\pi)$ is closed under unions of cardinality $\kappa$.
\end{lemma}

\section{Automorphisms of $\SP(\RR)/\Fin$}
\label{sec:reals}
	
In this section we define a cardinal characteristic of the continuum - the Cofinal Selection Number  - and use it to
show that a certain fragment of MA$_{\aleph_{1}}$ (a consequence of $\cov(\SM) > \aleph_{1}$, where $\cov(\SM)$ denotes
the covering number for the ideal of meager sets) implies that any automorphism of $\SP(\RR)/\Fin$ which is trivial on
all countable sets is trivial. By Theorem \ref{thm:trivial-kappa+}, it suffices to prove this  result with $\omega_{1}$
in place of $\RR$; as this makes no essential difference in the proof, we work with $\RR$. Veli\v ckovi\'c has shown
\cite{Velickovic.OCAA} that MA$_{\aleph_{1}}$ implies that any automorphism of $\SP(\omega_{1})/\Fin$ which is trivial
on all countable sets is trivial. His fragment of MA$_{\aleph_{1}}$ is different, corresponding roughly to adding
$\aleph_{1}$ many Cohen reals and then specializing an Aronszajn tree.

\begin{definition}
  Given $\Gamma\subseteq\SP(2^\omega)$, we let $\CSN(\Gamma)$ be the smallest cardinality of a family
  $\SF\subseteq (2^\omega)^\omega\times (2^\omega)^\omega$ such that
  \begin{enumerate}
    \item  for every $(f,g)\in\SF$, $\set{f(n)}{n < \omega}\cup \set{g(n)}{n < \omega}$ is dense in $2^\omega$,

    \item\label{cond:disjoint} for all pairs $(f,g), (f',g')$ from $\SF$, if $g \neq g'$, then
    \[
      \set{g(n)}{n < \omega}\cap \set{g'(n)}{n < \omega} = \emptyset,
    \]
    \item  for every $(f,g)\in\SF$ and $n < \omega$, $f(n) \neq g(n)$, and
    \item  for every set $A\in\Gamma$, the set
    \[
      \set{(f,g)\in\SF}{\exists^\infty n < \omega\;\; \card{A\cap\{f(n),g(n)\}} = 1}
    \]
    has cardinality smaller than that of $\SF$,
	\end{enumerate}
  if such a family $\SF$ exists. If no such family exists, we set $\CSN(\Gamma) = (2^{\aleph_{0}})^{+}$.
\end{definition}

It is not hard to see that $\CSN(\SP(2^{\omega})) = (2^{\aleph_{0}})^{+}$ (condition (\ref{cond:disjoint}) was included
to make this the case).

Consider the poset $\QQ$ with conditions $(\sigma,s)$, where $\sigma\in 2^{<\omega}$ and $s$ is a function mapping
into $2$, with domain $\set{(\sigma\rs n)^\frown \langle 1 - \sigma(n)\rangle}{n < \dom{\sigma}}$.  We define
$(\sigma,s) \le (\tau,t) \iff \sigma\supseteq\tau \land s\supseteq t$.  It is easy to see that $\QQ$ is isomorphic
to a dense subset of $\CC\times\CC$, where $\CC$ is Cohen forcing.  Given $p = (\sigma,s)\in \QQ$ we define
\[
  U_p = \bigcup \set{N_\tau}{\tau\in\dom{s}\land s(\tau) = 1}
\]
and
\[
  V_p = \bigcup \set{N_\tau}{\tau\in\dom{s}\land s(\tau) = 0}
\]
and, given $G\subseteq\QQ$, we set \[U_G = \bigcup\set{U_p}{p\in G}\] and \[V_G = \bigcup\set{V_p}{p\in G}.\]

\begin{lemma}
  \label{lemma:cohen}
  Let $X = \set{x_n}{n < \omega}$ and $Y = \set{y_n}{n < \omega}$ be subsets of $2^\omega$ such that $X\cup Y$ is dense
  and $x_n \neq y_n$ for all $n < \omega$.  Then if $G$ is $\QQ$-generic, there are infinitely many $n < \omega$ such
  that $U_G$ contains exactly one of $x_n,y_n$.
\end{lemma}
\begin{proof}
  Given $p\in\QQ$ we let $E_p$ be the set of $n \in \omega$ for which $U_{p}$ and $V_{p}$ each contain a member of
  $\{x_n,y_n\}$. We will show that for each $p\in\QQ$, there exist $q\le p$ and $n \not\in E_p$ with $n\in E_r$.  Let
  $p\in\QQ$ be given.  Since $X\cup Y$ is dense, there must be some $n$ such that at least one of $x_n$ or $y_n$ is in
  $[\sigma_p]$.  We consider the case $x_n\in [\sigma_p]$; the case $y_{n}  \in [\sigma_{p}]$ can be handled similarly.

  Suppose first that $y_n\not\in [\sigma_p]$; then $y_n\supseteq \tau$ for some $\tau\in\dom{s_p}$.  Let $\sigma_q$ be
  some extension of $\sigma_p$ such that $x_n\not\supseteq \sigma_q$; say $k$ is minimal such that $x_n(k)\neq
  \sigma_q(k)$.  Let $\nu = (\sigma_q\rs k)^\cat x_n(k)$.  Define $s_q$ so that $s_q(\nu) = 1 - s_p(\tau)$.

  Now suppose $y_n\in [\sigma_p]$.  Then we may find $\sigma_q$ extending $\sigma_p$ such that neither of $x_n,y_n$ are
  in $[\sigma_q]$.  Let $k$ and $\ell$ be minimal such that $x_n(k)\neq \sigma_q(k)$ and $y_n(\ell)\neq \sigma_q(\ell)$;
  let \[\tau = (\sigma_q\rs k)^\cat x_n(k)\] and \[\nu = (\sigma_q\rs \ell)^\cat y_n(\ell).\]  Define $s_q$ so that
  $s_q(\tau) = 0$ and $s_q(\nu) = 1$.
\end{proof}

Lemma \ref{lemma:cohen}, and the fact that $\QQ$ is isomorphic to a dense subset of $\CC\times\CC$, give the following.

\begin{corollary}
  \label{cor:sep}
  $\CSN(\mathbf{\Sigma^0_1}) \ge \cov(\SM)$.
\end{corollary}

Theorem \ref{thm:sep->trivial} is a variant of Theorem \ref{thm:trivial-kappa+}, restricted to the case $\kappa =
\omega$.  Theorem \ref{thm:sep->trivial} assumes triviality on countable sets, instead of  sets of cardinality
$\aleph_{1}$, as in Theorem \ref{thm:trivial-kappa+}, at the cost of assuming a weak fragment of Martin's Axiom.

\begin{theorem}
  \label{thm:sep->trivial}
  Assume $\CSN(\mathbf{\Delta^1_1}) > \omega_1$, and let $\pi$ be a cardinality-preserving automorphism of
  $\SP(2^\omega)/\Fin$ which is trivial on every countable subset of $2^\omega$.  Then $\pi$ is trivial.
\end{theorem}

\begin{proof}
  Suppose that $\pi$ is nontrivial and let $\pi^*$ be a bijective selector for $\pi$.  We may choose $\pi$ so that for
  all $\sigma,\tau\in 2^{<\omega}$, if $\sigma\subseteq \tau$ then $\pi^*(N_\sigma)\supseteq \pi^*(N_\tau)$, and if
  $\sigma\perp \tau$ then $\pi^*(N_\sigma)\cap \pi^*(N_\tau) = \emptyset$.

  For each $x\in 2^\omega$ and $n < \omega$, there is a unique $\sigma\in 2^n$ such that $x\in \pi^*(N_\sigma)$.
  Moreover, these $\sigma$'s form a branch through $2^\omega$, which we will call $h(x)$.  Thus $h : 2^\omega\to
  2^\omega$ is a function satisfying $h^{-1}(N_\sigma) = \pi^*(N_\sigma)$ for each $\sigma\in 2^{<\omega}$.  By
  Lemma~\ref{lemma:additivity}, it follows that $h^{-1}(B) \sd \pi^*(B)$ is countable, for every Borel set $B\subseteq
  2^\omega$. In particular, $h$ is countable-to-one.
	
  Let $Q$ be the set of $\sigma\in 2^{<\omega}$ such that $\pi$ is nontrivial on $N_\sigma$.
	
  \begin{claim}
    $Q$ is a perfect tree.
  \end{claim}
  \begin{proof}
    Clearly, every $\sigma\in Q$ has at least one extension in $Q$.  Suppose that for some $\sigma\in Q$, there is
    exactly one $x\in [Q]$ which extends $\sigma$.  Then $\pi$ is trivial on $N_\tau$ for every $\tau\supseteq\sigma$
    with $\tau\not\subseteq x$; hence by Lemma~\ref{lemma:closure2} (and our assumption that $\pi$ is trivial on
    countable sets), $\pi$ is trivial on their union, i.e.  $N_\sigma\sm\{x\}$.  But then clearly $\pi$ is trivial on
    $N_\sigma$.
  \end{proof}
  For each countable $A\subseteq 2^\omega$ we may fix a function $f_A : \pi^*(A)\to A$ such that for all $X\subseteq A$,
  $\pi([X]) = [f_A^{-1}(X)]$.
  \begin{claim}
    \label{claim:dense and not equal}
    Let $A\subseteq 2^\omega$ be countable.  Then there is a countable set $B\subseteq 2^\omega$ with $B\supseteq A$,
    and an infinite set $X\subseteq B\sm A$, such that $h''X\cup f_B''X$ is dense in $[Q]$, and for every $x\in X$,
    $h(x)\neq f_B(x)$.
  \end{claim}

  \begin{proof}
    Suppose otherwise.  Then there is a countable $A^*\subseteq 2^\omega$ such that for every countable $B\supseteq
    A^*$, there is some $\sigma\in Q$ such that the set of $x\in B\sm A^*$ with $h(x)\neq f_B(x)$ and $N_\sigma\cap
    \{h(x),f_B(x)\} \neq \emptyset$ is finite.  Pressing down, we may fix a $\sigma^*\in Q$ and a finite $F^* \subseteq
    2^\omega$ such that for all $B$ in some stationary subset of $[2^\omega]^\omega$, if $x\in B\sm A^*\cup F^*$, then
    whenever one of $h(x),f_B(x)$ is in $N_{\sigma^*}$, we have $h(x) = f_B(x)$.  In particular, if $x\in
    \pi^*(N_{\sigma^*}) = h^{-1}(N_{\sigma^*})$, then $h(x)\in N_{\sigma^*}$ and so $h(x) = f_B(x)$ as long as $x\in
    B\sm (A^*\cup F^*)$.  It follows that $\pi$ is trivial on $\pi^*(N_{\sigma^*})$, a contradiction.
  \end{proof}

  By applying Claim~\ref{claim:dense and not equal} repeatedly, we may find a $\subseteq$-increasing sequence $\langle
  A_\alpha : \alpha < \omega_1\rangle$ consisting of countable subsets of $2^\omega$, and infinite sets $X_\alpha
  \subseteq A_{\alpha+1}\sm A_\alpha$, such that for every $\alpha < \omega_1$, $h''X_\alpha \cup f_{A_{\alpha+1}}''
  X_\alpha$ is dense in $[Q]$, and $h(x) \neq f_{A_{\alpha+1}}(x)$ for every $x\in X_\alpha$. Since $h$ is
  countable-to-one, we may thin the sequence if necessary so that the sets $h[X_{\alpha}]$ are disjoint for distinct
  $\alpha$.  Applying the assumption that $\CSN(\mathbf{\Delta^1_1}) > \omega_1$ (and possibly thinning our sequence
  again), we may find a Borel set $B\subseteq [Q]$ such that for every $\alpha < \omega_1$, there are infinitely many
  $x\in X_\alpha$ such that $B$ contains exactly one of $h(x), f_{A_{\alpha+1}}(x)$.  Hence,
  \[
    (h^{-1}(B)\sd f_{A_{\alpha+1}}^{-1}(B))\cap (A_{\alpha+1}\sm A_\alpha)
  \]
  is infinite for every $\alpha < \omega_1$.  Since $f_{A_{\alpha+1}}^{-1}(B)\sd (\pi^*(B)\cap A_{\alpha+1})$ is finite
  for each $\alpha$, it follows that $\pi^*(B)\sd h^{-1}(B)$ is uncountable.  This is a contradiction.
\end{proof}

\begin{remark}\label{remark:dhyp}
  Since $\dd \ge \cov(\SM)$, if we replace ``$\CSN(\mathbf{\Delta^1_1}) > \omega_1$'' with ``$\cov(\SM) > \omega_1$'' in
  Theorem~\ref{thm:sep->trivial}, by Theorem~\ref{thm:balcar-frankiewicz:dd} we can then remove the assumption that
  $\pi$ is cardinality-preserving.  On the other hand, it follows from Theorem~\ref{thm:balcar-frankiewicz} that if
  $\pi$ and $\pi^{-1}$ are both trivial on every countable subset of $2^\omega$, then $\pi$ must be
  cardinality-preserving.
\end{remark}

\section{Isomorphisms between countable subalgebras}
\label{sec:qsets}

In~\cite[Theorem~3.1]{Geschke.shift}, Geschke showed that any isomorphism between countable subalgebras of
$\SP(\omega)/\Fin$ extends to a trivial automorphism, and in fact a trivial automorphism witnessed by a permutation of
$\omega$.  In this section we find a necessary and sufficient condition for an isomorphism between two countable,
atomless subalgebras of $\SP(\lambda)/\SI_\kappa$ to extend to a trivial automorphism, in the case where $\kappa$ has
uncountable cofinality.  We then use our result to examine automorphisms of $\SP(\omega_1)/\Ctble$ when a $Q$-set
exists.

Recall that any two countable, atomless Boolean algebras are isomorphic.  In particular, if $\SA$ is a countable,
atomless Boolean algebra, then $\SA$ is isomorphic to the Boolean algebra of clopen subsets of $2^\omega$, hence
$\SA$ is generated by elements $a_\sigma$ ($\sigma\in 2^{<\omega}$) such that $a_\sigma\meet a_\tau = 0$ for
$\sigma\perp \tau$, and $a_{\sigma^\cat 0}\join a_{\sigma^\cat 1} = a_\sigma$.  Note that in this case, if $\pi : \SA\to
\SB$ is an isomorphism, then the elements $b_\sigma = \pi(a_\sigma)$ ($\sigma\in 2^{<\omega}$) generate $\SB$, and
satisfy the same relations.

\begin{definition}
  If $\SA$ is a countable, atomless Boolean subalgebra of $\SP(\lambda)/\SI_\kappa$, then we say that a sequence
  $\bar{A} = \seq{A_\sigma}{\sigma\in 2^{<\omega}}$ of subsets of $\lambda$ is a \emph{nice sequence of representatives}
  for $\SA$ if
  \begin{itemize}
    \item  for every $\sigma\in 2^{<\omega}$, the sets $A_{\sigma^\cat 0}$ and $A_{\sigma^\cat 1}$ partition $A_\sigma$,
      and
    \item  the sequence $[A_\sigma]$ ($\sigma\in 2^{<\omega}$) generates $\SA$.
  \end{itemize}

  If $\bar{A}$ is a nice sequence of representatives for $\SA$, and $x\in 2^\omega$, then we set
  \[
    A_x = \bigcap\set{A_\sigma}{\sigma \subset x}
  \]
  and
  \[
    X(\bar{A}) = \set{x\in 2^\omega}{A_x\neq \emptyset}
  \]
\end{definition}
Note that if $\bar{A}$ is a nice sequence of representatives for $\SA$, then $[A_\emptyset]$ is the top element of
$\SA$, and
\[
  A_\emptyset = \bigcup\set{A_x}{x\in X(\bar{A})}
\]
The following lemma shows that this information serves as a sort of invariant for $\SA$.
\begin{lemma}
  \label{lemma:X-invariance}
  Suppose that $\SA$ is a countable, atomless Boolean subalgebra of $\SP(\lambda)/\SI_\kappa$, where $\cf{\kappa} > \omega$,
  and that $\bar{A}$ and $\bar{B}$ are nice sequences of representatives for $\SA$ with $A_\sigma \sim_\kappa B_\sigma$
  for each $\sigma\in 2^{<\omega}$.  Then there is an $S\in \SI_\kappa$ such that for every $x\in 2^\omega$,
  $A_x\sm S = B_x\sm S$.  Moreover, $X(\bar{A})  \sim_{\kappa} X(\bar{B})$.
\end{lemma}

\begin{proof}
  Let $S = \bigcup\set{A_\sigma\sd B_\sigma}{\sigma\in 2^{<\omega}}$; then $\card{S} < \kappa$, and for every $x\in
  2^\omega$, $A_x\sm S = B_x\sm S$.  Now, if $x\in X(\bar{A})\sd X(\bar{B})$, then it follows that
  $A_x\cup B_x\subseteq S$, since one of $A_x$ or $B_x$ must be empty.  But for each $\alpha\in S$, there is at most
  one $x\in 2^\omega$ such that $\alpha \in A_x$, and likewise, there is at most one $y\in 2^\omega$ such that
  $\alpha\in B_y$.
\end{proof}

Theorem \ref{thm:ctble-iso} characterizes when an isomorphism between two countable, atomless subalgebras of
$\SP(\lambda)/\SI_\kappa$ can extend to a trivial isomorphism.

\begin{theorem}
  \label{thm:ctble-iso}
  Let $\bar{A}$, $\bar{B}$ be nice sequences of representatives for countable, atomless Boolean subalgebras of
  $\SP(\lambda)/\SI_{\kappa}$, where $\cf{\kappa} > \omega$.  Then the following are equivalent.
  \begin{enumerate}
    \item  There is a trivial isomorphism from $\SP(A_{\emptyset})/\SI_{\kappa}$ to $\SP(B_{\emptyset})/\SI_{\kappa}$
    which sends $[A_\sigma]$ to $[B_\sigma]$, for every $\sigma\in 2^{<\omega}$.
    \item  $X(\bar{A}) \sim_{\kappa} X(\bar{B})$, and there is some $S\in \SI_\kappa$ such that for all $x\in 2^\omega$,
    $\card{A_x\sm S} = \card{B_x\sm S}$.
  \end{enumerate}
\end{theorem}

\begin{proof}
	Suppose that $f : E\to F$ is a bijection, where
	\begin{itemize}
  	\item $E \subseteq A_{\emptyset}$,
  	\item $F \subseteq B_{\emptyset}$,
  	\item $A_{\emptyset}\sm E$ and $B_{\emptyset}\sm F$ are both in $\SI_\kappa$,
  	\item $f''(A_\sigma \cap E)\sim_\kappa B_\sigma$ for all $\sigma\in 2^{<\omega}$.
  \end{itemize}
  Let $\bar{C} = \seq{f''(A_\sigma)}{\sigma\in 2^{<\omega}}$.  Then $\bar{B}$ and $\bar{C}$ satisfy the hypotheses of
  Lemma~\ref{lemma:X-invariance}, so $X(\bar{C}) \sim_\kappa X(\bar{B})$.  Since $f''(A_x\cap E) = C_x\cap F$ for all $x
  \in 2^{\omega}$, it follows that $X(\bar{C})\sim_\kappa X(\bar{A})$ as well, so $X(\bar{A}) \sim_\kappa X(\bar{B})$.
  Now let $S$ be as given in Lemma~\ref{lemma:X-invariance}, applied to $\bar{B}$ and $\bar{C}$, so that $C_x\sm S =
  B_x\sm S$ for all $x \in 2^{\omega}$.  Expanding $S$ if necessary (but preserving its membership in $\SI_{\kappa}$) we
  may assume that $S$ contains $A_{\emptyset}\sm E$ and $B_{\emptyset}\sm F$ and is closed under $f$ and $f^{-1}$.
  Then, for each $x \in 2^{\omega}$, $f$ maps $A_x \sm S$ to $B_x\sm S$. Since $f$ is a bijection it follows that
  $\card{A_x \sm S} = \card{B_x \sm S}$.

  Suppose now that $X(\bar{A}) \sim_{\kappa} X(\bar{B})$, and that there exists an $S\in \SI_\kappa$ such that
  $\card{A_x\sm S} = \card{B_x\sm S}$ for every $x\in 2^\omega$.  For each $$x\in X(\bar{A})\sd X(\bar{B}),$$ either
  $A_x = \emptyset$ or $B_x = \emptyset$, so \[A_x\sm S = B_x\sm S = \emptyset\] and \[A_x\cup B_x\subseteq S.\]  It
  follows that, if \[E = \bigcup\set{A_x}{x\in X(\bar{A})\cap X(\bar{B})}\] and \[F = \bigcup\set{B_x}{x\in
  X(\bar{A})\cap X(\bar{B})},\] then $A_\emptyset\sm S\subseteq E$ and $B_\emptyset\sm S\subseteq F$.  For each $x\in
  X(\bar{A})\cap X(\bar{B})$, choose a bijection $g_x : A_x\sm S\to B_x\sm S$.  Since for $x\neq y$ we have $A_x\cap A_y
  = B_x\cap B_y = \emptyset$, it follows that $g = \bigcup_{x\in X(\bar{A})\cap X(\bar{B})} g_x$ is a bijection from
  $E\sm S$ to $F \sm S$, such that $g''(A_\sigma\sm S) = B_\sigma\sm S$ for every $\sigma\in 2^{<\omega}$.  Then $g$ is
  as desired.
\end{proof}

\begin{remark}
  If $X\subseteq\RR$ has no isolated points then $X\cap N_\sigma$ ($\sigma\in 2^{<\omega}$) forms a nice sequence of
  representatives for a countable, atomless subalgebra $\SA_{X}$ of $\SP(\RR)/\Ctble$, and the set $X(\bar{A})$, where
  each $A_{\sigma}$ is $X\cap N_{\sigma}$, is exactly equal to $X$. If $X$ and $Y$ are two such sets, then there is an
  isomorphism from $\SA_{X}$ to $\SA_{Y}$ sending $[X\cap N_{\sigma}]$ to $[Y\cap N_{\sigma}]$ for every $\sigma$.  By
  Theorem \ref{thm:ctble-iso}, if $X \sd Y$ is uncountable, then there does not exist a trivial isomorphism from
  $\SP(X)/\Ctble$ to $\SP(Y)/\Ctble$ sending each set $[X\cap N_{\sigma}]$ to $[Y\cap N_{\sigma}]$.
\end{remark}

\begin{definition}\label{def:bqset}
  A set $X\subseteq 2^\omega$ is a \emph{$Q_B$-set} if for every $Y\subseteq X$, there is a Borel $B\subseteq 2^\omega$
  such that $B\cap X = Y$.
\end{definition}

The above is a weakening of the usual notion of a $Q$-set, where the set $B$ above is required to be $G_\delta$ and not
just Borel.  The reader can consult~\cite{Miller.special-sets, Miller.special-subsets} for properties of $Q$-sets.  We
note in particular that $\mathrm{MA}_{\kappa}$ implies all subsets of $\RR$ of size $\kappa$ are $Q$-sets (this is credited to 
Silver on page 162 of \cite{Martin-Solovay}), and that the
Ramsey forcing axiom $\SK_4$ implies that all subsets of $\RR$ of size $\omega_1$ are $Q$-sets
(\cite{Todorcevic-Velickovic.MAP}).  As for the difference between $Q$-sets and $Q_B$-sets, Miller
(\cite{Miller.length}) showed that if $X\subseteq\RR$ is a $Q_B$-set, then there is some $\alpha < \omega_1$ such that
every subset of $X$ is relatively $\mathbf{\Sigma^0_\alpha}$, whereas it is consistent that for every $2\le \alpha <
\omega_1$ there is an uncountable $X\subseteq\RR$ such that $\alpha$ is the minimal ordinal for which every subset of
$X$ is relatively $\mathbf{\Sigma^0_\alpha}$.
	
\begin{theorem}
  \label{thm:qsets}
  Suppose there exists a $Q_B$-set $X\subseteq 2^\omega$ of cardinality $\lambda$, where $\cf{\lambda} > \omega$.  Then
  the following are equivalent.
  \begin{enumerate}
    \item\label{cdn:nontrivial-aut}  There exists a cardinality-preserving, nontrivial automorphism of $\SP(\lambda)/\Ctble$.
    \item\label{cdn:inseparable-sets}  There exists a $Q_B$-set $Y\subseteq 2^\omega$ such that $X\sd Y$ is
    uncountable, and for every Borel $B\subseteq 2^\omega$, $\card{B\cap X} + \aleph_0 = \card{B\cap Y} + \aleph_0$.
  \end{enumerate}
\end{theorem}

\begin{proof}
  Assuming that~\eqref{cdn:nontrivial-aut} holds, we may fix a nontrivial cardinality-preserving isomorphism $\pi$ from
  $\SP(X)/\Ctble$ to $\SP(\lambda)/\Ctble$.  Choose a sequence of sets $\bar{A} = \seq{A_\sigma}{\sigma\in 2^{<\omega}}$
  such that $A_\emptyset = \lambda$, and, for all $\sigma \in 2^{<\omega}$,
  \begin{itemize}
    \item $\pi([N_\sigma \cap X]) = [A_\sigma]$,
    \item $A_\sigma$ is the disjoint union of $A_{\sigma^\cat 0}$ and $A_{\sigma^\cat 1}$.
  \end{itemize}
  Since $X$ is a $Q_B$-set, every subset of $X$ is equal to $C\cap X$ for some Borel $C$.  Since $\SP(\lambda)/\Ctble$
  is countably complete, $\pi([C \cap X]) = [D]$, where $D$ is the set built from $\bar{A}$ in the same way that $C$ is
  built from $\langle N_{\sigma} : \sigma \in 2^{<\omega}\rangle$.  Since $\pi$ is an isomorphism, every subset of
  $\lambda$ is, up to a countable set, a member of the $\sigma$-algebra generated by the sets $A_\sigma$ ($\sigma\in
  2^{<\omega}$). It follows that for each $y \in 2^{\omega}$, the set \[A_{y} = \bigcap\{ A_{y \upharpoonright n} : n
  \in \omega\}\] is countable, and that the set of $y \in 2^{\omega}$ for which $|A_{y}| \geq 2$ is countable. Let $Y$
  be the set of $y \in 2^{\omega}$ for which $|A_{y}| = 1$, and for each $y \in Y$, let $h(y)$ be the unique element of
  $A_{y}$. Then $h$ is injective, $h''(Y)$ is a cocountable subset of $\lambda$, and $h^{-1}(A_\sigma) = N_\sigma\cap Y$
  for each $\sigma \in 2^{<\omega}$.  We claim that $Y$ is the desired set.

  To see that $Y$ is a $Q_{B}$-set, fix $Z \subseteq Y$. Then there exists a $D$ in the $\sigma$-algebra generated by
  $\bar{A}$ such that $D \sd h''(Z)$ is countable. Then $h^{-1}(D) \sd Z$ is countable, and $h^{-1}(D)$ is equal to $B
  \cap Y$ for a Borel set $B \subseteq 2^{\omega}$ which is built from the sets $N_{\sigma}$ in the same way that $D$
  was built from the sets $A_{\sigma}$. Similarly, for each Borel set $B \subseteq 2^{\omega}$, letting $D$ be the set
  built from the sets $A_{\sigma}$ in the way that $B$ was built from the sets $N_{\sigma}$, $\pi([B \cap X]) = [D]$ and
  $h^{-1}(D) = B \cap Y$.  Since $\lambda \setminus h''(Y)$ is countable, \[|B \cap Y| + \aleph_{0} = |D| +
  \aleph_{0}.\] Since $\pi$ is cardinality-preserving \[|D| + \aleph_{0} = |C \cap X| + \aleph_{0}.\] Finally, if $X \sd
  Y$ were a countable set $S$, we could fix a bijection $b \colon X \setminus S \to h''(Y \setminus S)$ by setting
  $b(x)$ to be $h(x)$. Then as above, for each Borel set $B \subseteq 2^{\omega}$, $\pi([B \cap X]) = [h''(B\cap Y)] =
  [b''(B \cap X)]$. Again applying the fact that $X$ is a $Q_{B}$-set, this shows that $\pi$ is trivial.

  Now suppose that $Y \subseteq 2^{\omega}$ witnesses ~\eqref{cdn:inseparable-sets}. Define $\pi : \SP(X)/\Ctble \to
  \SP(Y)/\Ctble$ by
  \[
    \pi([C\cap X]) = [C\cap Y]
  \]
  where $C$ ranges over the Borel sets.  We claim that
  \begin{enumerate}[(i)]
    \item  $\pi$ is well-defined,
    \item  $\pi$ is an isomorphism, and
    \item  $\pi$ is cardinality-preserving and nontrivial.
  \end{enumerate}

  All of these follow easily from our assumptions, except perhaps the claim that $\pi$ is nontrivial.  Suppose then that
  $S$ and $T$ are countable subsets of $2^{\omega}$, and that $f$ is bijection from $X \setminus S$ to $Y \setminus T$
  such that $\pi([A]) = [f''(A \setminus S)]$ for all $A \subseteq X$.  Since $X \sd Y$ is uncountable there are
  uncountably many $x \in X \sm S$ such that $f(x) \neq x$. Moreover, we may fix incompatible $\sigma$, $\tau$ in
  $2^{<\omega}$ such that the set
  \[
    Z = \set{x\in N_\sigma\cap (X\sm S)}{f(x)\in N_\tau}
  \]
  is uncountable.  Then there is a Borel set $B$ such that $Z = B \cap X = (B\cap N_\sigma) \cap X$, so by the
  definition of $\pi$,
	\[
    \pi([Z]) = [(B\cap N_\sigma)\cap Y]
  \]
  On the other hand, $f''(Z) \subseteq N_\tau$, hence $f''(Z)\cap N_\sigma = \emptyset$.  This contradicts our
  assumption that $\pi([Z]) = [f''(Z)]$.
\end{proof}

\begin{remark}
  \label{rmk:calkin}
  We do not know whether it is consistent with ZFC that there exists a nontrivial automorphism of
  $\SP(\omega_1)/\Ctble$.  On the other hand, Theorem~1 of~\cite{Farah-McKenney-Schimmerling.SC} shows that the
  analogous result for Calkin algebras holds under the assumption that $2^{\omega_1} = \omega_2$.  More precisely, let
  $\SB$ denote the C*-algebra of bounded, linear operators on a Hilbert space of dimension $\omega_1$, and let $\SJ$ be
  its (closed, two-sided, $*$-) ideal of operators with separable range.  Then $2^{\omega_1} = \omega_2$ implies there
  are $2^{\omega_2}$-many automorphisms of $\SB/\SJ$.
\end{remark}
	
Notice that condition~\eqref{cdn:inseparable-sets} above fails whenever the union of two $Q_B$-sets of cardinality
$\omega_1$ is also a $Q_B$-set.  Combined with Theorem~\ref{thm:trivial-kappa+}, we obtain the following as a corollary.
	
\begin{corollary}
  \label{cor:PR-mod-ctble}
  Suppose that there is a $Q_B$-set of cardinality $\omega_1$, and the union of any two $Q_B$-sets of cardinality
  $\omega_1$ is a $Q_B$-set.  Then every automorphism of $\SP(\RR)/\Ctble$ is trivial.
\end{corollary}
	
The cardinal characteristic $\mathfrak{q}_{0}$ is the least cardinality of a subset of $2^{\omega}$ which is not a
$Q$-set (see \cite{Banakh-Machura-Zdomskyy}, for instance).  Following \cite{Zapletal.DSTDF}, we let $\mathfrak{z}$ be the
least cardinality of a subset of $2^{\omega}$ which is not a $Q_{B}$-set.  Then $\mathfrak{q}_{0}$ is clearly at most
$\mathfrak{z}$. We note that $\mathfrak{q}_{0}$ is at most $\mathfrak{d}$ (see \cite{Banakh-Machura-Zdomskyy}; we do not
know if the same holds with $\mathfrak{z}$ in place of $\mathfrak{q}_{0}$.) and that
MA$_{\aleph_{1}}$ implies that $\mathfrak{q}_{0} = 2^{\aleph_{0}}$, by the result of Silver mentioned after Definition \ref{def:bqset}. 
	
Corollary~\ref{cor:PR-mod-ctble} and Theorem~\ref{thm:lifts} give the following.

\begin{corollary}\label{cor:zcons}
  If $\mathfrak{z} > \aleph_{1}$ then each of the following hold.
  \begin{itemize}
    \item Every automorphism of $\SP(\RR)/\Ctble$ is trivial.
    \item Every cardinality-preserving automorphism of $\SP(\RR)/\Fin$ is trivial on a cocountable set.
  \end{itemize}
\end{corollary}	

In conjunction with the main result of \cite{Shelah-Steprans.2}, we see that $\mathfrak{z} > \aleph_{1}$ implies that
every cardinality-preserving automorphism of $\SP(\lambda)/\Fin$ is trivial on a cocountable set, for every $\lambda$
less than the least strongly inaccessible cardinal.  Veli\v ckovi\'c \cite{Velickovic.OCAA} has shown that
$\MA_{\aleph_{1}}$ is consistent with the existence of a nontrivial automorphism of $\SP(\omega)/\Fin$.

\begin{remark}\label{rem:PR-mod-ctble}
  Corollary~\ref{cor:PR-mod-ctble} applies to automorphisms $\pi$ which may not be induced by automorphisms of
  $\SP(\RR)/\Fin$.  We do not know if the hypothesis of either of Theorem \ref{thm:sep->trivial} and Corollary~\ref{cor:PR-mod-ctble} 
	implies the other. However, as
  $\CSN(\SP(2^{\omega})) = (2^{\aleph_{0}})^{+}$, $\CSN(\mathbf{\Delta^1_1})$ is at least $\mathfrak{z}$.
\end{remark}

\section{Fixed points}\label{sec:fixpoints}

If $\pi$ is an automorphism of a Boolean algebra of the form $\SP(\lambda)/\SI_\kappa$, then a \emph{fixed point} of
$\pi$ is a set $A\subseteq\lambda$ such that $\pi([A]) = [A]$. A fixed point $A$ is \emph{nontrivial} if $A$ and
$\lambda \setminus A$ both have cardinality at least $\kappa$. By the $<\cf\kappa$-completeness of
$\SP(\lambda)/\SI_\kappa$, the set of ($\sim_{\kappa}$-classes of) fixed points of such a $\pi$ is a (possibly trivial)
$<\cf\kappa$-complete subalgebra of $\SP(\lambda)/\SI_\kappa$.

\begin{lemma}
  Suppose that $\kappa \leq \lambda$ are infinite cardinals, and that $\pi$ is an automorphism of
  $\SP(\lambda)/\SI_{\kappa}$. Let $\pi^{*}$ be a selector for $\pi$.  Let $\eta$ be an infinite regular cardinal
  not equal to $\cf\kappa$, and suppose that $\langle A_{\alpha} : \alpha < \eta \rangle$ is a sequence of subsets of
  $\lambda$ such that
  \begin{enumerate}
    \item  for all $\alpha < \beta < \eta$,
    \[
      \card{(\pi^*(A_\alpha)\cup (\pi^*)^{-1}(A_\alpha)) \setminus A_{\beta}} < \kappa
    \]
    \item\label{eq:almost-continuity}  for all $\beta < \eta$,
    \[
      \card{\bigcup\set{A_\alpha}{\alpha < \beta}\sm A_\beta} < \kappa
    \]
  \end{enumerate}
  Then $\bigcup\{ A_{\alpha} : \alpha < \eta\}$ is a fixed point of $\pi$.
\end{lemma}

\begin{proof}
  Let $B = \bigcup\set{A_\alpha}{\alpha < \eta}$.  We want to see that $\pi^*(B) \sim_\kappa B$.  Suppose first that
  $|B\sm \pi^*(B)| \ge \kappa$.  We claim that there is some $\alpha < \eta$ such that $\card{A_\alpha\sm \pi^*(B)} \ge
  \kappa$.  If $\eta < \cf{\kappa}$, then this follows directly; on the other hand, if $\eta > \cf{\kappa}$, then (using
  the regularity of $\eta$) there is some $\alpha < \eta$ such that
  \[
    \card{\bigcup\set{A_\beta\sm \pi^*(B)}{\beta < \alpha}} \ge \kappa
  \]
  in which case we have $\card{A_\alpha\sm \pi^*(B)} \ge \kappa$ by~\eqref{eq:almost-continuity}.  
  Now fix some $X\subseteq A_\alpha\sm \pi^*(B)$ with cardinality exactly $\kappa$.  Then
  \[
    \card{(\pi^*)^{-1}(X)\sm A_{\alpha+1}} < \kappa
  \]
  and
  \[
   \card{(\pi^*)^{-1}(X)\cap B} < \kappa
  \]
  hence $\card{(\pi^*)^{-1}(X)} < \kappa$, a contradiction. Supposing instead that
  \[
    \card{\pi^{*}(B)\sm B} \ge\kappa
  \]
  we get that $\card{B \sm (\pi^{*})^{-1}(B)} \geq \kappa$, and we can run the argument just given with $(\pi^{*})^{-1}$
  in place of $\pi^{*}$ to obtain another contradiction.
\end{proof}

Summarizing, we get the following. Part (\ref{bfcon}) of the theorem uses Theorem \ref{thm:balcar-frankiewicz}.  The
proof of part (\ref{bfcon3}) breaks into two cases, one where $\cf\kappa$ is uncountable, and one where $\cf\kappa =
\aleph_{0}$.  %where $|A| < \kappa$ or $|A| > \aleph_{1}$.

\begin{theorem}
  \label{thm:fixpoint-summary}
  Suppose that $\kappa \leq \lambda$ are infinite cardinals, and that $\pi$ is an automorphism of
  $\SP(\lambda)/\SI_{\kappa}$.
  \begin{enumerate}
    \item The set of $\sim_{\kappa}$-classes of fixed points $\pi$ is a $<\cf\kappa$-complete subalgebra of
    $\SP(\lambda)/\SI_{\kappa}$.

    \item If $\eta$ is a regular cardinal not equal to $\cf\kappa$, and $\langle A_{\alpha} : \alpha <\eta \rangle$ is a
    sequence of fixed points of $\pi$ such that for all $\beta < \eta$,
    \[
      \card{\bigcup\set{A_\alpha}{\alpha < \beta} \sm A_\beta} < \kappa,
    \]
    then $\bigcup\set{A_\alpha}{\alpha < \eta}$ is a fixed point of $\pi$.

    \item\label{bfcon} If $\kappa$ is regular and $\lambda > \kappa$, then for every $A \subseteq \lambda$, there is a fixed point $B
    \subseteq \lambda$ such that $A \subseteq B$ and $|B| \leq |A| + \kappa^{+}$.
    %\item\label{bfcon2} If $\kappa = \omega$ and $\lambda$ is uncountable, then for every $A \subseteq \kappa$, there is a
    %fixed point $B %\subseteq \lambda$ such that $A \subseteq B$ and $|B| = |A| + \aleph_{1}$.

    \item\label{bfcon3} If $\pi$ is cardinality-preserving and $\kappa$ is uncountable, then for every $A \subseteq
    \lambda$, there is a fixed point $B \subseteq \lambda$ such that $A \subseteq B$ and $|B| = |A|$.
    %\item If $\pi$ is cardinality-preserving and $ \lambda > \kappa = \omega$, then for every $A \subseteq \kappa$, there
    %is a fixed %point $B \subseteq \lambda$ such that $A \subseteq B$ and $|B| = |A| \cdot \aleph_{1}$.
  \end{enumerate}
\end{theorem}

Theorem \ref{thm:fixpoint-summary} gives the following corollary.

\begin{corollary}
  If $\lambda > \kappa$ are infinite cardinals, and $\pi$ is an automorphism of $\SP(\lambda)/\SI_{\kappa}$, then $\pi$
  has nontrivial fixed points if at least one of the following holds:
  \begin{itemize}
    \item $\kappa$ is regular and $\lambda > \kappa^{+}$,
    %\item $\kappa = \omega$ and $\lambda > \aleph_{1}$,
    \item $\pi$ is cardinality preserving and $\kappa$ is uncountable.
  \end{itemize}
\end{corollary}

\begin{remark}
  Chodounsk\'y, Dow, Hart and de Vries have shown (\cite{CDHdV}) that if the algebras $\SP(\omega)/\Fin$ and
  $\SP(\omega_{1})/\Fin$ are isomorphic, then there exists a nontrivial automorphism of $\SP(\omega)/\Fin$. A simpler,
  previously known argument (see \cite{Nyikos}) uses the trivial automorphism of $\SP(\omega)/\Fin$ induced by the
  (upwards or  downwards) shift, getting that  if $\SP(\omega)/\Fin$ and $\SP(\omega_{1})/\Fin$ are isomorphic, then
  there exists an automorphism of $\SP(\omega_{1})/\Fin$ without nontrivial fixed points.  An easy argument shows that
  every trivial automorphism of $\SP(\omega_1)/\Fin$ has a club of ordinal fixed points.
\end{remark}

Section \ref{sec:ladders} considers ordinal fixed points of cofinality $\kappa$ for automorphisms of Boolean algebras of
the form $\SP(\kappa^{+})/\SI_{\kappa}$.

\section{Ladder systems}
\label{sec:ladders}

A \emph{ladder} on a limit ordinal $\alpha$ is a cofinal subset of $\alpha$ whose ordertype is the cofinality of
$\alpha$ (we do not require here that the subset be closed). If $S$ is a set of limit ordinals, a \emph{ladder system on
$S$} is a sequence $\seq{L_\alpha}{\alpha\in S}$ such that each $L_\alpha$ is a ladder on the corresponding $\alpha$.  A
ladder system $\seq{L_\alpha}{\alpha\in S}$ satisfies \emph{$\kappa$-uniformization} (for a given cardinal $\kappa$) if
for every sequence of functions $f_\alpha : L_\alpha\to\kappa$ ($\alpha\in S$), there is a function $F : \sup(S) \to
\kappa$ such that for all $\alpha\in S$,
\[
  \set{\beta\in L_\alpha}{F(\beta) \neq f_\alpha(\beta)} \in \SI_{\kappa}.
\]

Here we show that the existence of a cardinality-preserving automorphism of $\SP(\kappa^{+})/\SI_{\kappa}$ without
ordinal fixed points of cofinality $\kappa$ (where $\kappa$ is a regular cardinal) gives rise to a ladder system on a
club subset of $\kappa^{+}$ which satisfies $2$-uniformization (which is easily seen to be equivalent to $\gamma$-uniformization, for 
any $\gamma < \kappa$). 

Given an automorphism $\pi$ of
$\SP(\lambda)/\SI_{\kappa}$, for infinite cardinals $\kappa \leq \lambda$, we say that $\beta \in \lambda$ is a
\emph{closure point} of $\pi$ if for all $\alpha < \beta$, $\pi([\alpha]) < [\beta]$ and $\pi^{-1}([\alpha]) <
[\beta]$.\footnote{See subsection~\ref{subsec:notation} for the meaning of the order $<$ in this context.}
If $\lambda > \kappa$ and $\pi$ is cardinality-preserving, then the set of closure points of $\pi$ is a club subset of
$\lambda$. It is easy to see that every closure point whose cofinality is not $\cf \kappa$ is a fixed point (this does
not require that $\pi$ is cardinality-preserving, or that $\kappa$ is regular).

\begin{theorem}
  \label{thm:2-unif}
  Suppose that
  \begin{itemize}
    \item $\kappa$ is a regular cardinal,
    \item $\pi$ is a cardinality-preserving automorphism of $\SP(\kappa^{+})/\SI_{\kappa}$,
    \item $C$ is the set of closure points of $\pi$, and
    \item $S \subseteq \kappa^{+}$ is the set of $\alpha \in C$ which are not fixed points of $\pi$.
  \end{itemize}
  Then there exist $S_{0}, S_{1}$ such that $S = S_{0} \cup S_{1}$, and such that $S_{0}$ and $S_{1}$ each support a
  ladder system satisfying $2$-uniformization.
\end{theorem}

\begin{proof}
  Let
  \[
    S_{0} = \set{\alpha \in S}{[\alpha] \not\leq \pi([\alpha])}
  \]
	and let
  \[
    S_{1} = \set{\alpha \in S}{\pi([\alpha]) \not\leq [\alpha]}
  \]
  Then clearly $S = S_{0} \cup S_{1}$.  Let $\pi^*$ be a bijective selector for $\pi$.  For each $\alpha\in S_0$, let
  $L^{0}_\alpha = \alpha\sm\pi^*(\alpha)$, and for each $\alpha\in S_1$, let $L^{1}_\alpha =
  \alpha\sm(\pi^*)^{-1}(\alpha)$.  Then for each $i \in \{0,1\}$ and each $\alpha \in S_{i}$,
	\begin{itemize}
	  \item $|L^{i}_{\alpha}| = \kappa$,
	  \item $|L^{i}_{\alpha} \cap \beta| < \kappa$ for all $\beta < \alpha$.
	\end{itemize}
  The second of these follows from the fact that $\alpha$ is a closure point of $\pi$, and from the fact that
  $\card{(\pi^{*})^{-1}(L^{0}_{\alpha}) \cap \alpha} < \kappa$ in the case where $\alpha \in S_{0}$ and
  $\card{\pi^{*}(L^{1}_{\alpha}) \cap \alpha} < \kappa$ in the case where $\alpha \in S_{1}$.  It follows that each
  $L^{i}_\alpha$ is a ladder on the corresponding $\alpha$.

  Now suppose we are given $2$-colorings $f^{i}_\alpha : L^{i}_\alpha\to 2$ for each pair $(\alpha, i)$ with $i \in
  \{0,1\}$ and $\alpha \in S_{i}$.  For each such pair $(\alpha, i)$ let $a^{i}_\alpha = (f^{i}_{\alpha})^{-1}(\{1\})$.
	For each $\alpha\in S_0$, put
	\[
    b_\alpha^0 = (\pi^*)^{-1}(a_\alpha)\cap ((\pi^*)^{-1}(\alpha)\sm \alpha)
  \]
  and for each $\alpha\in S_1$, put
	\[
    b_\alpha^1 = \pi^*(a_\alpha)\cap (\pi^*(\alpha)\sm \alpha).
  \]
  Notice that, for each $i \in \{0,1\}$, $b_\alpha^i\cap b_\beta^i = \emptyset$ for distinct $\alpha,\beta\in S_i$.

  Let $B_i = \bigcup\set{b_\alpha^i}{\alpha\in S_i}$, for each $i < 2$.  Define $A_{0} = \pi^*(B_0)$ and $A_{1} =
  (\pi^*)^{-1}(B_1)$.  For each $i \in \{0,1\}$, let $F_{i}$ be the characteristic function of $A_{i}$.  If $\alpha\in
  S_0$, then
	\[
    B_0\cap ((\pi^*)^{-1}(\alpha)\sm \alpha) = b_\alpha^0
  \]
  hence $A_{0}\cap L_\alpha \sim_{\kappa} a_\alpha$.  Similarly, if $\alpha\in S_1$, then
  \[
    B_1\cap (\pi^*(\alpha)\sm \alpha) = b_\alpha^1
  \]
  so $A_{1}\cap L_\alpha \sim_{\kappa} a_\alpha$.  It follows that $F_{0}\rs L_\alpha \sim_{\kappa} f_\alpha$ for all
  $\alpha\in S_{0}$ and $F_{1}\rs L_\alpha \sim_{\kappa} f_\alpha$ for all $\alpha\in S_{1}$.

\end{proof}

The following theorem of Devlin and Shelah then shows that the existence of a cardinality-preserving automorphism of
$\SP(\omega_{1})/\Fin$ without nontrivial ordinal fixed points implies that $2^{\aleph_{0}} = 2^{\aleph_{1}}$.
	
\begin{theorem}(Devlin-Shelah \cite{Devlin-Shelah})
  \label{thm:devlin-shelah}
  Suppose that $\{ S_{\alpha} : \alpha < \omega_{1}\}$ is such that
	\begin{itemize}
  	\item each $S_{\alpha}$ is a subset of $\omega_{1}$ supporting a ladder system satisfying 2-uniformization,
  	\item the diagonal union of $\set{S_{\alpha}}{\alpha < \omega_{1}}$ contains a club subset of $\omega_{1}$.
	\end{itemize}
	Then $2^{\aleph_{0}} = 2^{\aleph_{1}}$.
\end{theorem}

After a simple modification of $\pi^*$, we see that $\pi^*$ moves the ladders $L_\alpha$, for $\alpha\in S_0$, to
disjoint sets; and $(\pi^*)^{-1}$ moves $L_\alpha$ for $\alpha\in S_1$ to disjoint sets.  It follows that they satisfy
uniformization properties stronger than 2-uniformization (but not comparable, as far as we know, with
$\kappa$-uniformization). For instance, they each satisfy the following property : for any partition of $S_{i}$ into
sets $\set{T_{\alpha}}{\alpha < \gamma}$ (for some $\gamma \leq \kappa^{+}$) there exist sets $\set{K_{\alpha}}{\alpha <
\gamma}$ such that
\begin{itemize}
  \item for all $\alpha < \gamma$ and all $\beta \in T_{\alpha}$, $\card{L_{\beta} \setminus K_{\alpha}} < \kappa$,
  \item for every sequence of functions $f_{\alpha} \colon K_{\alpha} \to 2$ $(\alpha < \gamma)$ there exists a function $F \colon
  \kappa^{+} \to 2$ such that, for each $\alpha < \gamma$, $F \restrict K_{\alpha} \sim_{\kappa} f_{\alpha}$.
\end{itemize}

\section{Open questions}
\label{sec:questions}
	
We collect here various open questions related to the material in this paper, some of which have been mentioned above,
and some of which have been asked by others.  First, we ask for various types of automorphisms.

\begin{question}
  Are any of the following consistent with ZFC?
  \begin{enumerate}[(a)]
    \item  There exists an uncountable cardinal $\lambda$ and an automorphism of $\SP(\lambda)/\Fin$ which is not
    trivial on any cocountable set.
    \item  There exists an uncountable cardinal $\lambda$ and an automorphism of $\SP(\lambda)/\Fin$ which is not
    trivial on any uncountable set.  (By \cite{Shelah-Steprans.2}, $\lambda$ would have to be at most $2^{\aleph_{0}}$.)
    \item  There exists an infinite cardinal $\kappa$ and a nontrivial automorphism of $\SP(\kappa^+)/\SI_\kappa$ which
    is trivial on all sets of cardinality $\kappa$.
    \item  There exists an infinite cardinal $\kappa$ such that all automorphisms of $\SP(\kappa)/\SI_\kappa$ are
    trivial, but there is a nontrivial automorphism of $\SP(\kappa^+)/\SI_\kappa$.
    \item  There exist infinite cardinals $\kappa < \lambda$ and a nontrivial automorphism of $\SP(\lambda) /
    \SI_{\kappa^+}$ which is trivial on all sets of cardinality $\kappa^+$.  (By Theorem~\ref{thm:trivial-kappa+},
    $\lambda$ would have to be bigger than $2^\kappa$.)
    \item  There exists an automorphism of $\SP(\omega_1)/\Fin$ with no nontrivial fixed points.  (What if the
    automorphism is required to be cardinality-preserving?  In this case, we would have to have $2^{\aleph_0} =
    2^{\aleph_1}$, by Theorems~\ref{thm:2-unif} and~\ref{thm:devlin-shelah}.)
    %\item  There exists an uncountable cardinal $\lambda$ and a nontrivial automorphism of $\SP(\lambda)/\Ctble$.
		\item  There exist uncountable cardinals $\kappa \leq \lambda$ and a nontrivial automorphism of $\SP(\lambda)/\SI_{\kappa}$ (what if 
		$\kappa = \aleph_{1}$?).
    \item  There exists an uncountable cardinal $\lambda$ and an outer automorphism of the Calkin algebra on the Hilbert
    space of dimension $\lambda$.  (See Remark~\ref{rmk:calkin} or~\cite{Farah-McKenney-Schimmerling.SC} for definitions
    and more information.)
  \end{enumerate}
\end{question}

We also ask about the Katowice Problem, question~\eqref{q:KP} below, and its relation to automorphisms.  The
reader is referred to~\cite{KPH-HdV, Chodounsky.SQSD} for more on the Katowice Problem and related
questions.  We note in particular that, in~\cite{Chodounsky.SQSD}, Chodounsk\' y has constructed a model of ZFC where
most of the known consequences of a positive answer to Question~\eqref{q:KP} hold.
\begin{question}
  \begin{enumerate}[(a)]
    \item\label{q:KP}  (Turzanski) Is it consistent with ZFC that the Boolean algebras $\SP(\omega)/\Fin$ and $\SP(\omega_1)/\Fin$
    are isomorphic?
    \item  Is it consistent with ZFC that there is an isomorphism from $\SP(\omega_1)/\Fin$ to $\SP(\omega)/\Fin$ which
    is trivial on all countable sets?
    \item  Is it consistent with ZFC that there exists an automorphism $\pi$ of $\SP(\omega_1)/\Fin$ such that $\pi([A])
    = [B]$ for no infinite $A\subset B\subseteq \omega_1$ with $\omega_1\sm B$ infinite?
    \item  Does the existence of an isomorphism $\SP(\omega)/\Fin\simeq \SP(\omega_1)/\Fin$ imply that there is a
    nontrivial automorphism of $\SP(\omega_1)/\Ctble$?  (Since such an isomorphism implies there is an uncountable
    $Q$-set, by Theorem~\ref{thm:qsets} it is enough to ask whether such an isomorphism implies there exist two
    uncountable $Q_B$-sets, with uncountable difference, which intersect the same Borel sets uncountably.)
    \item Does the existence of an isomorphism between $\SP(\omega)/\Fin$ and $\SP(\omega_{1})/\Fin$ imply that $\mathfrak{z} = \aleph_{1}$?
  \end{enumerate}
\end{question}

\bibliography{biblio}{}
\bibliographystyle{plain}

\end{document}